\documentclass[12pt]{amsart} 
\usepackage{amsmath,amstext,amsfonts,amssymb} 

\newtheorem{theo}{\bf Theorem}[section] 
\newtheorem{conj}[theo]{\bf Conjecture}

\newtheorem{defi}[theo]{\bf Definition}

\newtheorem{prob}[theo]{\bf Problem} 
\newtheorem{prop}[theo]{\bf Proposition} 
\newtheorem{quest}[theo]{\bf Question} 
\newtheorem{rem}[theo]{\bf Remark} 

\newcommand{\ctl}{\centerline} 
\newcommand{\equi}{\Leftrightarrow}  
\newcommand\ifff{if and only if } 
\newcommand{\la}{\langle} \newcommand{\ra}{\rangle} 
\newcommand{\nd}{{\text{ and }}} 
\newcommand{\noi}{{\noindent}} 

\newcommand{\abs}[1]{\lvert#1\rvert} 
\newcommand{\ov}{\overline} 
\newcommand{\sm}{\smallsetminus} 

\DeclareMathOperator{\Co}{Co} 
\DeclareMathOperator{\Oo}{O} 
\DeclareMathOperator{\PSL}{PSL} 
\DeclareMathOperator{\Qa}{Qa} 
\DeclareMathOperator{\Qb}{Qb} 
\DeclareMathOperator{\Qc}{Qc} 
  
\DeclareMathOperator{\rk}{rk} 

\renewcommand{\a}{\alpha} 
\newcommand{\g}{{\gamma}} 
\newcommand{\Lb}{{\Lambda}} 
\newcommand{\lb}{{\lambda}} 
\newcommand{\tht}{{\theta}} 
\newcommand{\vp}{{\varepsilon}} 

\newcommand{\A}{{\mathbb A}} 
\newcommand{\D}{{\mathbb D}} 
\newcommand{\E}{{\mathbb E}} 
\newcommand{\Q}{{\mathbb Q}} 
\newcommand{\R}{{\mathbb R}} 
\newcommand{\Z}{{\mathbb Z}} 

\newcommand{\cC}{{\mathcal C}} 
\newcommand{\cE}{{\mathcal E}}

\begin{document} 
\title[Lattices and equiangular lines] 
{Families of Equiangular Lines and Lattices} 
\author[J. Martinet] 
{Jacques Martinet} 
 

\email{Jacques.Martinet@math.cnrs.fr} 

\begin{abstract} 
We construct Euclidean lattices whose sets of minimal vectors 
support some large equiangular families of lines, 
using notably reduction modulo~$2$ of lattices.  
We also consider some related problems, and answer a question raised 
by Greaves ([G], Subsection~1.3.2). 
\end{abstract} 

\maketitle 

\section*{Introduction 
} \label{secintro} 

In this note we consider constructions of sets of equiangular lines 
in a Euclidean space $E$ derived from Euclidean lattices. We postpone 
to Section~\ref{secbasic}  the necessary prerequisites on Euclidean lattices 
together with the statement of various general results, stating here only 
the following theorem: 

\begin{theo} \label{thintro} 
Let $\Lb$ be an even (integral) lattice in a Euclidean space $E$, 
of minimum~$m$, and let $x_0\in\Lb$ of norm 
(the square of the length) \hbox{$2m\!-\!2$}.  
Then the set vectors of norm \hbox{$2m+2$} in $\Lb$ orthogonal to $x_0$ 
and congruent to $x_0$ modulo~$2$ 
supports an equiangular family of lines of common angle $\arccos\frac 1{m+1}$. 
\end{theo} 

\noi 
In the rest of the introduction we solely recall some basic facts 
and questions concerning families of equiangular lines. 
Since the publication in 1973 of Lehmens-Seidel seminal paper \cite{L-S}, 
there has been a large literature on the subject, among which I would like 
to quote the collective work \cite{JTYZZ} (Ann. Math., 2021). 
We refer to Greaves' survey~\cite{G} for the results mentioned below 
without a reference. We shall in particular note that for a set of a given 
rank~$n$, the numbers of lines $t=n$, $n+1$, $2n$ and $\frac{n(n+1)}2$ 
play somewhat special r\^oles. 

\medskip 

Consider a set $X$ of $t$ lines of rank~$n$ with common angle 
$\tht\in[0,\frac{\pi}2]$, with which we associate the parameter 
$\a=\arccos\tht$, and assume that $t>n$. We then have $\tht\ne\frac{\pi}2$, 
thus $0<\a<1$. Equip each line with a norm-$1$ vector, and consider 
their Gram matrix $G(X)$, with entries $1$ on its diagonal and $\pm\a$ 
off its diagonal. The matrix 

\ctl{$S(X):=\frac1{\a}(I_n-G(X))$} 

\smallskip\noi 
is a {\em Seidel matrix for $X$\/}. It depends on the ordering and orientations 
of the lines in~$X$, but its spectrum solely depends on~$X$, and in particular 
its smallest eigenvalue $\lb$, equal to $-\frac 1{\a}<-1$. 
It is immediate that if all $\pm\a$ are $+\a$ (resp. $-\a$), 
we have $\lb=-1$, which must be excluded (resp, $\lb=-(t-1)$, $t=n+1$ and 
$\a=\frac 1n$): in other words, for an acute (resp. obtuse) 
equiangular family of {\em vectors\/}, we have $t\le n$ (resp. $t\le n+1$, 
with $\a=\frac 1n$ if equality holds). 
Also we see that $t\ge n+1$ implies that the angle may take only finitely 
many values. 

The {\em absolute bound\/} ({\em Gerzen's theorem\/}) is the inequality 
$t\le\frac{n(n+1)}2$, which holds because the orthogonal projections to 
the line of $X$ are independent symmetric automorphisms. 

The importance of the value $2n$ (or $2n+1$) 
comes from {\em Neuman's theorem\/}: 
if $t>2n$, $\frac 1{\a}$ is an odd integer; compare Theorem~\ref{thintro}. 

\medskip 

In general one considers the maximum number of lines in an equiangular set 
without taking into account their rank: in terms of the definition below, 
one considers values or estimations for $t_n$, {\em not\/} for $t'_n$. 

\begin{defi} \label{defitmax} {\rm 
Let $t_n$ be the maximal cardinality of an equiangular set contained 
in an $n$-dimensional space (resp. spanning an $n$-dimensional space). 
{\small 
[Thus we have $t'_n\le t_n=\max_{k\le n}\,t_k$.]}} 
\end{defi} 

The exact values of $t_n$ are known up to $n=17$. As we shall see, we have 
$t'_n=t_n$ for $n\le 7$ and $14\le n\le 17$, but in the range $7<n<14$, 
we have the {\em strict inequality\/} $t'_n<t_n$. 
It is also known (cf. Neuman's theorem above) that $t_n$ is strictly 
larger than $2n$ except if $n=2$, $3$ $4$, $5$ or~$14$. 
This suggests the following double question: 

\begin{quest} \label{questt_n} 
For which dimensions $n$ is it true that $t'_n=t_n$? that $t'_n>2n$? 
\end{quest} 

A neighbour question closer to the subject of this note is: 

\begin{quest} \label{questlat} 
For which dimensions is the bound $t'_n$ attained on the set of minimal 
vectors of a Euclidean lattice ? 
\end{quest} 
I have put emphasis on $t'_n$ rather than on~$t_n$ because I believe 
that $t'_n$ presents irregularities as a function of~$n$ which throw some 
light on individual properties of dimensions. In this respect 
it is worth considering the absolute bound $\frac{n(n+1)}2$. 
This is known to be attained for $n=2,3,7,23$, but on no other dimensions. 
These dimensions are somewhat special, related to the existence of the 
regular hexagon and icosahedron for $n=2,3$, and to the lattices 
$\E_8$ and Leech's $\Lb_{24}$ of dimension $n+1$ for $n=7,23$. 
\newline{\small 
[This can be related to the double transitivity of the groups 
$S_3$, $A_5$, $\Oo_7(2)\sim S_6(2)$ and $\Co_3$ in degrees $3,6,28$ and $276$, 
respectively. For $n=3$, one identifies the action of $\A_5$ on the diagonals 
of a regular icosahedron as its action on its $5$-Sylow subgroups; 
lattices for $n=2,7,23$ will be described later.]}

Recent works have shown the special behaviour of the ``magic'' dimensions 
$8$ and $24$ for kissing number and sphere packing problems. 
This suggest that the absolute bound could be strict in all other dimensions. 
The following  question of maximality could be considered in place of it. 
Consider a configuration $X$ of equiangular lines which is maximal 
in the $n$-dimensional space its spans. Say that $X$ is 
{\em universally maximal\/} if $X$ remains maximal when embedded 
in dimension~$n+1$ (or any larger dimension). 
The four dimensions above are universally maximal: trivially for $n=2$, 
because $t'_n=t_n=t_{n+1}$ for $n=3,7,23$. 

\begin{prob} \label{probuniversal} 
For which dimensions $n$ does there exist a universally maximal configuration 
of $t'_n$ lines? 
\end{prob} 

Here are two more results we shall need later. 
First the {\em relative bound\/} (see \cite{G}, Th.\,1.9): 
if $t\le \frac 1{a^2}$, then $t\le\frac{n(1-\a^2)}{1-n\a^2}$. 

Next (\cite{JTYZZ}): if $\a=\frac 1{2k-1}$ ($k\ge 2$), then we have 
$t'_n=\lfloor\frac{k(n-1)}{k-1}\rfloor$ for $n$ large enough. 

\medskip 

In Section~\ref{secbasic} we explain how to apply the theory of lattices
modulo~$2$ as developed in  \cite{Ma1} and \cite{Ma2} to the construction 
of equiangular systems of lines. 
In Section~\ref{secroot}  we consider root lattices 
(certain lattices of minimum~$2$) and their duals, 
and lattices of minimum~3 derived from them, 
and in Section~\ref{sec3-4}, lattices of minimum $4$ and $5$. 

\medskip 
\noi{\bf 
Acknowledgments.} 
The author thanks Bill Allombert for his help in working with {\em PARI.GP\/}.

\section{Basic Results} \label{secbasic} 

Let $E$ be a  Euclidean space of dimension~$n$, and $\Lb$ be a (full) 
lattice in~$E$. We denote by $m$ its minimum, by $S$ the set 
of its minimal vectors and set $s=\frac 12\abs{S}$. More generally, 
given $r>0$, let 

\smallskip\ctl{ 
$S_r=\{x\in\Lb\mid x\cdot x=r\}$ and let $s_r=\frac 12\abs{S_r}$\,;} 

\smallskip\noi 
thus, $S=S_m$ and $s=s_m$. We also set $N(x)=x\cdot x$ 
(the {\em norm of~$x$\/}, the square of the usual $\|x\|$). 
We say that $\Lb$ is {\em integral\/} if all scalar products on $\Lb$ 
are integral, and then {\em even\/} if all norms are even, 
{\em odd\/} otherwise. Note that $\Lb$ is integral \ifff it is contained 
in its dual $\Lb^*=\{x\in E\mid \forall\,y\in\Lb,\,x\cdot y\in\Z\}$. 

\smallskip 

Consider vectors $x,y\in\Lb$, with $x\ne 0$, $y\ne 0$ and $y\ne\pm x$, 
and such that $y\equiv x\!\!\mod2$. Set 

\smallskip\ctl{$e=\frac{y-x}2\,\nd\,f=\frac{y+x}2$\,.} 

\smallskip\noi 
Then $e,f$ are nonzero, the set $\{\pm e,\pm f\}$ only depends on 
$\{\pm x,\pm y\}$, and we have 

\ctl{$x=-e+f\,\nd\,y=e+f$\,.} 

\smallskip\noi 
Moreover, in an equality  $\pm y=\pm x+2u$, $u$ is one of $\pm e$ or~$\pm f$.

\begin{prop} \label{propbasic} 
{\small\rm(\cite{Ma1}, \S~1)} 
With the notation above, we have 
\begin{enumerate} 
\item 
$N(x)+N(y)\ge 4m$. 
\item 
If $\Lb$ is integral, $N(y)\equiv N(x)\!\!\mod4 $.
\item 
If $(N(x)+N(y)=4m$, then $e$ and $f$ are minimal and $x\cdot y=0$. 
\end{enumerate} 
\end{prop} 

\begin{proof} 
(1) We have 

\ctl{ 
$N(x)+N(y)=\frac12\big(N(y-x)+N(y+x)\big)=2\big(N(e)+N(f)\big)\ge 4m$} 

\noi 
since $e$ and $f$ are nonzero. 

(2) We then have 

\ctl{ 
$N(y)-N(x)=N(e+f)-N(-e+f)=4 e\cdot f\equiv 0\!\!\mod 4$\,.} 

(3) We then have $N(e)+N(f)=2m$, hence $N(e)=N(f)=m$ 
\linebreak 
and $x\cdot y=N(f)-N(e)=0$. 
\end{proof} 

Proposition~\ref{propbasic} shows that a list of representatives of classes 
modulo~$2$ of smallest possible norms must include vectors of norm $N\le 2m$, 
and that the corresponding classes then include a unique pair $\pm x$ 
if $N<2m$, and at most $n$ such pairs if $N=2m$. 
Among known examples are the celebrated lattices $\E_8$ and $\Lb_{24}$. 

\begin{prop} \label{propscalarbound} 
Let $m'\ge m$ and $m''\ge m'$, and let $x,y,y'\in\Lb$ such that 
$N(x)=m'$, $N(y)=N(y')=m''$, $y\equiv y'\equiv x\mod 2$, and $y'\ne\pm y$, 
Then $\abs{y\cdot y'}\le m''-2m$. 
\end{prop} 

\begin{proof} 
  Write as in Proposition~\ref{propbasic}

 \smallskip\ctl{ 
 $x=-e+f,y=e+f\,\nd\,x=-e'+f',y'=e'+f'$\,.}  

\smallskip\noi 
Calculating $N(y-y')$ in two ways we obtain the equalities 

\smallskip\ctl{ 
$2m''-2y\cdot y'=4N(e-e')\ge 4m$} 

\noi(since $e'=e$ implies $y'=y$), hence 

\ctl{
$y\cdot y'=m''-2N(e-e')\le m''-2m$\,,} 

\noi 
since $N(e\pm e')\ge m$.

Calculating similarly  $N(y+y')$, we obtain 

\smallskip\ctl{ 
$2m''-2y\cdot y'=4N(-2x+f+f')\ge 4m$} 

\noi(since $x=f+f'$ implies $y'=-y$), hence 

\ctl{
$y\cdot y'=m''-2N(e-e')\le m''-2m$\,.} 
\end{proof}

\begin{theo} \label{thequiang} 
Assume that $\Lb$ is integral of even minimum. 
Let $x_0\in\Lb$ of norm $2m-2$. Consider the set 
$\cE(x_0,2m+2)=\{\pm y_1,\dots,\pm y_k\}$ of those vectors of norm $2m+2$ 
congruent to $x_0$ modulo~$2$. 
Then the $y_i$ support an equiangular family of lines of angle 
$\arccos\frac 1{m+1}$ and rank $\le n-1$. 
{\rm\small 
[By an abuse of language, we accept sets of less than three lines.]} 
\end{theo} 

\begin{proof} 
We have $\abs{y_i\cdot y_j}\le 2$ by Proposition~\ref{propscalarbound}. 
Next by the calculation done in the proof of this proposition 
we have $y_i\cdot y_j=m''-N(e-e')$ 
with $m''=2m+2$, and since $N(e)=N(e')=m$, 
we have $N(e-e')=2m-2e\cdot e'$, which implies 

\smallskip\ctl{$y_i\cdot y_j=-2\,m+2+4\,e\cdot e'\equiv 2\mod 4$\,,} 

\smallskip\noi 
hence $y_i\cdot y_j=\pm 2$. 

The bound for the rank comes from Proposition~\ref{propbasic},\,(4). 
\end{proof}

\begin{rem} 
\begin{enumerate} 
\item 
{\small\rm 
The restriction to {\em even\/} minima is essential. 
Indeed, if $m$ is odd, we may apply Proposition~\ref{propbasic} 
to the even sublattice $\Lb_{e v}$ of $\Lb$, of minimum $m_{e v}\ge m+1$. 
Since $m'+m''$ is strictly smaller than $4(m+1)$, 
$\cE(x_0,2m+2)$ is then empty. 
\item 
We may also clearly restrict ourselves to irreducible lattices.} 
\end{enumerate}\end{rem}

The proposition below, extracted from Section~5 of \cite{Ma2}, 
shows that the set of equiangular lines 
constructed in Theorem~\ref{thequiang} can be realized as the set 
of minimal vectors of an integral (relative) lattice of minimum~$m+1$.

\begin{prop} \label{prop_m+1} 
Let $x_0\in\Lb$ of norm $m'<2m$, let $\cC$ be its class modulo~$2$, 
let $L_0=\cC\cup2\Lb$ and let $L=L_0\cap x_0^\perp$. Set $m''=4m-m'$, 
and assume that $\cE(x_0,m'')$ is not empty. 
Then $L$ is a lattice with invariants 

\smallskip 
\ctl{$\dim L=n-1,\ \min L=m'',\ \nd\ S(L)=\cE(x_0,m'')$\,.} 
\end{prop} 

\begin{proof} 
We have $\cC=x_0+2\Lb$ and $2x_0\in2\Lb$, hence $L_0$ is a lattice 
(containing $2\Lb$ to index~$2$). By Proposition~\ref{propbasic},  
the first minimum of the norm on~$\cC$ is $m'$, attained uniquely 
at $\pm x_0$, and since $\cE$ is not empty, its second minimum 
is $m"$, attained exactly on $\cE$. 
Since $\min 2\Lb=4m>m''$, these are the first two minima of $L_0$, 
and since $\cE$ is orthogonal to $x_0$, we have $\min L=m''$ 
and $S(L)=\cE$. 
\newline{\small 
[The calculation of $\det(L)$ is carried out in \cite{Ma2}, Section~5.]} 
\end{proof} 

\begin{rem} \label{remsqrt2} {\small\rm 
The vectors in $\Lb'$ are sums of minimal vectors. This shows that its 
even part that $\Lb'_{e v}$ is generated by sums $e+e'$, $e,e'\in S(\Lb')$. 
Easy calculations then show that on $\Lb'_{e v}$, all scalar products 
are even and all norms are divisible by~$4$. 
Hence $\frac1{\sqrt 2}\Lb'_{e v}$ is an even (integral) lattice.} 
\end{rem} 

From an algorithmic viewpoint, listing the $y_i$ in Theorem~\ref{thequiang} 
from vectors of norm $2m+2$ can need lengthy computations. 
The following proposition allows us to complete this list using only minimal 
vectors. This also makes easy the calculation of the rank of~$\cE$. 
In the following proposition we keep the notation of \ref{thequiang}, setting  
moreover 

\smallskip\ctl{\{$S_0=\{x\in S \mid x_0\cdot x=m-1$\,.\}} 

\begin{prop} \label{propS_0} 
The map 

\smallskip\ctl{ 
$x\mapsto x_0-2x\,:\,S_0\to\Lb$} 

\smallskip\noi 
induces a bijection of $S_0$ onto $\cE$, and we have 
%
$\rk\cE=\rk S_0-1$. 
\end{prop} 

\begin{proof} 

The first assertion follows from the calculation of $x_0\cdot \frac{y-x}2$ 
and $N(\frac{y-x}2)$ for $y\in S_{2m+2}$. 

Let $r=\rk\cE$, let $y_1,\dots,y_r$ be $r$ independent vectors in~$\cE$, 
and let $x_i=\frac{x_0-y_i}2$. We have 

\smallskip\ctl{ 
$\la y_1,...,y_r\ra=\la y_1,y_2-y_1,\dots,y_r-y_1\ra= 
\la y_1,x_2-x_1,\dots,x_r-x_1\ra$\,.} 

\smallskip\noi 
Now $\la S_0\ra$ contains the $x_0-x_i$ 
(because $-y_i=x_0-2(x_0-x_i)$) and $x_0-2x_1$ ($=y_1$). Hence 
$\la S_0\ra=\la x_0,y_1,\dots,y_r\ra$\,, 
and since the $y_i$ but not $x_0$ are orthogonal to $x_0$, 
we have $\rk S_0=r+1$. 
\end{proof} 

Let $v$ be a nonzero vector in $E$ and let $H:=(\R v)^\perp$. 
The orthogonal projection to $H$ (or along $v$) of $x\in E$ is 

\smallskip\ctl{$p(x)=x-\frac{v\cdot x}{v\cdot v}\,v$\,.} 

\smallskip\noi 
Scalar products and norms of projections are given by the formulae 
{\small 
$$p(x)\cdot p(y)=x\cdot y-\frac{(v\cdot x)(v\cdot y)}{N(v)}  \nd 
N(p(x))=N(x)-\frac{(v\cdot x)^2}{N(v)}\,.\eqno{(\pmb\ast)}$$ 
} 
Note that for $x\in S$ and $x\ne\pm v$, we have 
$\abs{v\cdot x}\le \frac{N(v)}2$, since we have 
$N(v\pm x)\ge m$, hence $\mp 2(v\cdot x)\le N(v)$. 

\smallskip 
Consider a lattice $\Lb$. Then $p(\Lb)$ is a lattice \ifff $v\in\Lb$, 
and we may assume that $v$ is primitive 
(because $p$ only depends on the line $\R v$). 
We can then construct bases $(v_1,\dots,v_n)$ for $\Lb$ with $v_1=v$, 
so that $(p(v_2),\dots,p(v_n)$ is a basis for $p(\Lb)$. 

In the setting of Theorem~\ref{thequiang} projections along $x_0$ 
preserve the norms on $\cE$, and shall be used to construct lattices 
whose minimal vectors support equiangular families of lines. 
However I cannot give {\em a priori\/} a general procedure 
to find minimal vectors in projections. 
Note that if $\abs{v\cdot x}$,takes values $\a_1<\dots<\a_t$ 
on $S(\Lb)\sm\{\pm v\}$, 
the corresponding norms of the projections occur in the reverse order, 
and  $\a_t=\frac{N(v)}2$ needs $N(v)<4m$ by $(\pmb\ast)$. 
We shall consider the projections of the $x\in S(\Lb)$ such that 
$v\cdot x=\pm\a_t$, in particular when $N(v)=2m-2$ and $\a_t=\frac{N(v)}2$, 
since we then obtain equiangular families of line by applying 
Theorem~\ref{thequiang}. 
It will often happen that $S(p(\Lb))$ is the subset of $p(S(\Lb))$ 
of the $p(x)$ with $v\cdot x=\frac{N(v)}2$. This will be checked 
for all lattices we shall construct in the forthcoming sections. 

We state below a proposition which gives us some precisions on projections. 
Give $v$ we consider a partition $S^+\cup S^-$ of $S$ 
for which $x\in S^+$ needs $v\cdot x\ge 0$ (the choice of $x\in S^+$ 
among $\pm x$ is well-defined only when $v\cdot x\ne 0$). 
Using formulae $(\ast$), we easily prove: 

\begin{prop} \label{propproj} 
Consider a lattice $\Lb$ of minimum~$m$, a vector $v\in\Lb$ of norm~$2m-2$ 
and the orthogonal projection $p$ along~$v$, and assume that $S(p(\Lb)$  
is contained in~$p(S(\Lb))$. Let $x$ and $y\ne\pm x$ in~$S^+$, 
not colinear with $v$. Then we have 

\smallskip\ctl{$\frac{m-3}4\le x\cdot y\le\frac{3m-1}4$\,.} 

\smallskip 
In particular if $m=2$ (resp. $m=4$) we have $x\cdot y\in\{0,1\}$ 
\linebreak 
(resp. $x\cdot y\in\{1,2\}$.  
\qed 
\end{prop} 

\section{Root Lattices and their duals} \label{secroot} 

{\em Root lattices\/} are integral lattices generated by norm-$2$ vectors. 
Theses are orthogonal sums of irreducible lattices, isometric to exactly 
one of $\A_n$ ($n\ge 1)$, $\D_n$ ($n\ge 4$) or $\E_n$ ($n=6,7,8$, 
the definition of which we recall below; see \cite{Ma}, Chapter~4 for details. 

Inside the lattice $\R^{n+1}$ (resp. $\R^n$), equipped with its canonical 
basis $(\vp_0,\dots,\vp_n)$ (resp. $(\vp_1,\dots,\vp_n)$, Let 

\smallskip\ctl{\small 
$\A_n=\{x\in\Z^{n+1}\mid \sum_{i=0}^{i=n}x_i=0\}\nd%
\D_n=\{x\in\Z^n\mid \sum_{i=1}^{i=n}x_i\equiv0\mod 0\}$\,.} 

\medskip\noi 
In $\R^8$ let $e=\frac{\vp_1+\dots+\vp_8}2$, set $\E_8=\D_8\cup(\D_8+e)$, 
and define $\E_7$ and $\E_8$ by successive sections orthogonal 
to $\vp_7-\vp_8$ and $\vp_6-\vp_7$. 

\smallskip 
For a lattice $\Lb$ in the list above the automorphism group acts 
transitively on its set of minimal vectors, among which we may choose 
arbitrarily $x_0$ when applying Theorem~\ref{thequiang}. Let us choose 
$x_0=\vp_0-\vp_1$, $\vp_1-\vp_2$ and $e$ when $\Lb=\A_n$, $\D_n$ and $\E_n$, 
respectively. 

For $\Lb=\A_n$, we have 

\smallskip\ctl{\small 
$S_0=\{-\vp_0+\vp_i,\vp_1-\vp_i,\,i \ge 2\}\ \nd\  
\pm\cE=\{\vp_0+\vp_1-2\vp_i,\,i\ge 2\}$\,.} 

\smallskip 
For $\Lb=\D_n$, we have 

\smallskip\ctl{\small 
$S_0=\{-\vp_1\pm\vp_i,\vp_2\pm\vp_i,\,i\ge 3\}\ \nd\  
\pm\cE=\{\vp_1+\vp_2\pm\vp_i,\,i\ge 3\}$.} 

\smallskip 
For $\Lb=\E_8$, $S_0$ consists of $28\times 2=56$ vectors, 
the $\vp_i+\vp_j,\,1\le i<j\le 8$ and the $28$ vectors obtained by 
negating $6$ basis vectors in~$e$. The $28$ vectors in $\cE$ (up to sign) 
are then obtain by permutations of 

\smallskip\ctl{ 
$\frac{3\vp_1+3\vp_2-\vp_3-\vp_4-\vp_5-\vp_6-\vp_7-\vp_8}2$\,.}  

\smallskip\noi 
The results for $\E_7$ and $\E_6$ are obtained by sections of~$\E_8$. 
Alternatively we could also have used the sets of minimal vectors 
in projections of $\A_n,\D_n,\E_n$. 

Summarizing, we obtain: 

\begin{prop} \label{equirootlat} 
Let $\Lb$ be an irreducible root lattice of dimension~$n$. 
Then the equiangular family of Theorem~\ref{thequiang} is of rank~$r=n-1$ 
and contains $t$~lines according to the following data: 

\ctl{ 
$\Lb=\A_n$\,: $t=r$\,; \quad 
$\Lb=\D_n$\,: $t=2r-2$\,;} 
\ctl{ 
$\Lb=\E_8$: $t=28$\,; \quad  
$\Lb=\E_7$\,: $t=16$\,; \quad  
$\Lb=\E_6$: $t=10$.} 
\noi{\small\rm  
[
$\D_{n+1}$ provides the asymptotic bound for $\a=\frac 13$, 
valid for all $n\ge 15$.]} 
\qed 
\end{prop} 

As for the duals, we may discard $\D_4^*\sim\D_4$, $E_8^*=E_8$ 
and $\D_n^*,\,n\ge 5$ (because $S(D_n^*)=S(\Z^n)$). 
The set $S(\A_n^*)$ supports the equiangular set of {\em vectors\/} 
with common angle $-\arccos\frac 1n$, and $S(\E_7^*)$ is is the projection 
of $\cE$ attached to~$\E_8$. The lattice $\E_6^*$ does not apply directly 
to equiangular systems of lines.
\newline{\small 
[The function $\abs{x\cdot y}$ takes two values on non-proportional 
$x,y\in S(\E_6^*)$. We shall return to such lattices in \cite{Ma3} 
in connection with the theory of {\em strongly regular graphs\/} 
(with $S(\E_6^*)$ we can recover the {\em Schl\"afli graph\/}, 
attached to the system of lines on a non-singular cubic surface).]} 

\smallskip 

The values of $t_n$ are known up to $n=17$ (see~\cite{G}). For $n=2$ to $7$, 
$t_n$ is equal to $3,6,6,10$ and $16$, respectively, 
and we have $t_n=28$ for $7\le n\le 14$. 

\begin{prop} \label{prop2-13} 
\begin{enumerate} 
\item 
We have $t'_n=t_n$ for $2\le n\le 7$, and except for $n=3$, $t'_n$ is attained 
on the set of minimal vectors of a lattice. 
\item 
For $n=3$, the maximum number of lines defined by a lattice is~$4$, 
attained uniquely on the configuration of $S(\A_3^*)$. 
\item 
For $8\le n\le 13$,  $t'_n$ is strictly smaller than $t_n$.  
\end{enumerate}\end{prop} 

\begin{proof}  
(1) This is clear for $n=2$ and~$3$. For $n=4,5,6,7$, consider the projections 
of $\D_5$, $\E_6$, $\E_7$ and $\E_8$ (Proposition~\ref{prop2-13}). 

(2) 
This results from the classification of {\em minimal classes\/} for~$n=3$ 
(see \cite{Ma}, Theorem~9.2): one checks that a lattice with $s\ge 4$ 
must have a system of minimal vectors containing $e_1,e_2,e_3,e_1+e_2+e_3$ 
with equal scalar products $e_i\cdot e_j$. 

(3) 
Let $n$ as above, Assume that there exists an $n$-dimensional system 
of $s=28$ equiangular lines. Let $\tht$ be their common angle and let 
$\a=\arccos\tht$. 
We then have $s>2n$, so that by Neumann's theorem 
(\cite{G}, Theorem~1.16), we have $\a=\frac 1m$ for some odd integer $m\ge 3$. 
If $m\ge 5$, the {\em relative bound\/} $s\le\frac{1-\a^2}{1-s\a^2}$ 
(see \cite{G}, Theorem~1.9) implies $s\le 2n$, a contradiction. 
Hence we have $m=3$, and a theorem of Kao-Yu's (\cite{K-Y}) implies $n\le 27$. 
(If $n\le 14$, a system of $28$ equiangular lines with angle $\arccos\frac 13$ 
comes from dimension~$7$.) 
\end{proof}

\medskip 

An other way to construct lattices of minimum~$3$ consists in viewing them 
as lattices containing to index~$2$ their even sublattice. 
For further use we consider more generally lattices of minimum $m\ge 3$ odd. 
To state the results below we introduce some notation. 
Given a lattice $L$ and $a>0$, we denote by ${}^a L$ the group $L$ equipped 
with the scalar product $a\,(x\cdot y)$. 
(Thus we have $a\,(x\cdot y)\simeq\sqrt{a}\,L$.) 
Given an integral lattice $L$, 
$L_{e v}$ denotes the {\em even part of $L$\/}, i.e., the set of $x\in L$ 
having an even norm. If $L$ is odd, we have $[L:L_{e v}]=2$ 
and $L=\la L_{e v},e\ra$ where $e\in L$ is any vector of odd norm. 

\begin{prop} \label{propoddmin} 
Let $\Lb$ be a lattice of dimension $n\ge 2$ and of odd minimum $m\ge 3$, 
generated by its minimal vectors. Then $L:={}^{1/2}\Lb_{e v}$ is an even lattice 
of minimum $m'\ge\frac m2$, and either we have $s(\Lb)\le n$,  
or $L$ has a basis of vectors of norm~$m-1$. 
\end{prop} 

\begin{proof} 
First note that $\Lb_{e v}$ is generated by the sums $x+y$, $x,y\in S(\Lb)$, 
with $y\ne x$ since $2x=(x+y)+(x-y)$. 
We have $N(x+y)=2(m\pm 1)\equiv 0\mod 4$, which shows that $L$ is even. 

Choose a half-system $e_1,\dots,e_s$ of minimal vectors of $\Lb$ such that 
$e_1\cdot e_i=+1$ for $i\ge 2$ and that $e_1,\dots,e_n$ are independent, 
and set $e'_i=e_1-e_i,i\ge 2$. 
Suppose first that all $e_i\cdot e_j,i<j$ are equal to $+1$
and that $s>n$. Then we may write $x:=e_{n+1}$ as a $\Q$-linear 
combination of $e_1,\dots,e_n$, say, $x=\sum\lb_k\,e_k$. 
Calculating $x\cdot e_i$ for $i\le n$, we obtain 

\ctl{  
$1=x\cdot e_i=m\lb_i+\sum_{i\ne j}\lb_j=(m-1)\lb_i+\sum_i\,\lb_j$\,,} 

\smallskip\noi  
which implies that the $\lb_i$ have the common value $\lb=\frac 1{n+m-1}$. 

From $x=\lb\sum e_i$ we deduce that $m=x\cdot x=\lb\,n$, 
which implies 

\smallskip\ctl{$\frac 1{n+m-1}=\frac m n\ \equi\ (m+n)\,(n-1)=0$\,,} 

\noi a contradiction. 


\noi{\small 
[Lattices with $e_i\cdot e_j=+1$ and $s=n$ exist, and are unique 
up to isometry. They can be represented by the Gram matrix $M$ with entities
$M_{i,i}=m-1$ and $M_{i,j}=\frac{m-1}2$ if $j\ne i$, {\em except\/} 
$M_{1,1}=m+1$ and $M_{1,2}=M_{2,1}=0$.]} 

\smallskip 

Otherwise, let $i>1$ and $j>i$ such that $e_i\cdot e_j=-1$. 
If $i>n$, write $e_i$ as a $Q$=linear combination of $e_k,k\le n$. 
There must be at least three nonzero components, so that exchanging 
$i$ with some $i'\le n$, me may assume that $i\in[2,n]$, that me may then 
exchange with~$2$. We thus reduce ourselves to the case when $i=2$, 
and then we may similarly assume that~$j=3$. Finally we check 
that $e_2+e_3,e'_2,\dots,e'_n$ is a basis for $\Lb_{e v}$, 
which completes the proof of the dichotomy between the two cases above. 
\end{proof} 

Note that given $L$ we can reconstruct $\Lb$ from $L$ by the formula 

\smallskip\ctl{ 
$\Lb={}^2L_1\text{ with } L_1=\la L,\frac v 2\ra\text{ and } v\in L 
\text{ of norm }2m$\,.} 

\smallskip 

Using Proposition~\ref{propoddmin} together with the classification 
of root lattices we now prove: 

\begin{theo} \label{thmin3} 
Let $\Lb$ be an $n$-dimensional lattice of minimum~$3$ whose set $S$ 
of minimal vectors supports an equiangular family of lines 
{\em of rank~$n$\/}.  
Then except if $n=5,6$ or~$7$, the maximal number $s_n$ of lines 
is equal to $2(n-1)$, and except if $n=2$, attained uniquely on the projection 
of $\D_{n+1}$ described in Proposition~\ref{equirootlat}. 
\end{theo} 

\begin{proof} The first case of Proposition~\ref{propoddmin} needs 
$n\ge 2(n-1)$, hence $n=2$. 

Let now $n\ge 3$. Then $\Lb$ is of the form $\la L,\frac v 2\ra$ 
where $\Lb$ is a root lattice and $v$ a norm-$6$ vector 
which is not congruent modulo~$2$ to a shorter vector. 
This condition eliminates $\E_6$ and $\E_8$, and implies 
$n\ge 5$ (resp. $n\ge 6)$ if $\Lb=\A_n$ (resp. $\D_n$), 
and because $S$ must be of rank~$n$, $n=5$ (resp.~$6$). 
Finally we are left with $\A_5$, $\D_6$ and $\E_7$, which accounts 
for the exceptional dimensions $5,6,7$. 

From now on we assume that $L=L_1\perp L_2$ and $v=\frac{x+y}2$, 
$x\in L_1$, $y\in L_2$. The vectors congruent to $\frac v 2$ modulo~$l$ 
are of the form $\frac{x+2z,y+2t}2$ with z,t in $L$. We have 

\ctl{ 
$N(x+2z)=N(x)\equi x\cdot z+N(z)=0$\,,} 

\noi 
and if $N(z)\ge N(x)$ (a condition which is automatic if $N(x)=2$), 
this needs $z=-x$.

We may assume that $N(x)=2$ and $N(y)=4$, and since $x'=x+2z$ reduces to  
$x'=\pm x$, $\dim S=n$ needs $\dim L_1=1$, i.e., $L_1=\A_1$, and then 
either $L\simeq \A_1\perp\A_1\perp\A_1$ or $L2$ is irreducible. 
The former case accounts for the theorem in dimension~$3$. 
We are now left with $L=\A_1+\A_k,k\ge 3$, $L=\A_1+\D_k,k\ge 4$ 
and $L=\A_1+\E_k,k=6,7,8$ ($n=k+1$), the convenient $y+2t$ not $\pm y$ 
being obtained with $N(t)=2$ and $y\cdot t=-2$. 

If $L=\A_1\perp\A_k$ we may take $y=\vp_0+\vp_1-\vp_2-\vp_3$ 
and must take $k=3$. We obtain $s=6$, which accounts for the theorem 
in dimension~$4$. 

If  $L=\A_1+\D_k$ we may take (a) $y=2\vp_1$ or (b) $y=\vp_1+\vp_2+\vp_3+\vp_4$ 
(in the same orbit if $n=4$). Then $y+2t$ is one of $\pm\vp_i$ 
($2(n-1)$ solutions, $\rk\,S=n$) or $\pm\vp_1\pm\dots$ with an even number 
of minus signs ($8$ solutions, $\rk\,S=5$). 

If $L=\A_1\perp\E_{n-1}$, we perform direct calculations. 
The maximal value of $\rk\,S$ is then $6,7$ and $9$ for $n=7,8$ and $9$, 
respectively. In the latter cases $S$ generates a sublattice of index~$2$ 
in~$\Lb$, isometric to the lattice constructed with $\A_1\perp\D_8$> 

We have proved that fir $n\ne 2,5,6,7$, the maximum of $s$ is $2(n-1)$,
attained on a unique lattice (up to isometry) generated by its minimal vectors. This completes the proof of Theorem~\ref{thmin3}. 
\end{proof} 

\noi{\small 
[As a Gram matrix for the projection of $\D_{n+1}$ we may choose 
$A=(a_{i,j})$ with entries $3$ on the diagonal and $1$ off the diagonal, 
except $a_{1,2}=a_{2,1}=-1$.]} 

\smallskip 

It might well be that the methods of \cite{K-Y} could prove better upper bounds 
than $t'_n\le 27$ for some~$n$. 
In the other direction, we have the lower bounds $t'_n\ge 2(n-1)$ coming 
from $\D_{n+1}$, and we shall see in Section~\ref{sec3-4} that we have 
$t'_{13}\ge 26$, attained with a system of angle $\arccos\frac 15$. 
Experimentation based on Proposition~\ref{propoddmin} applied to various
``classical'' lattices of minimum~$4$ in dimensions $8$ to~$12$ always 
produced lattices with $s<2(n-1)$. 
This supports the following conjecture: 

\begin{conj} \label{conj8-13} 
For $8\le n\le 13$ we have $t'_n\le 2n$, and $t'_{13}=26$. 
\end{conj} 

The upper bound $2n$ allows that some $t'_n$ may be attained on angles not 
of the form $\arccos\frac 1m$, $m$~odd, like in dimension~$3$.

\section{Lattices of minimum 4 and 5} \label{sec3-4} 

In this section we concentrate on lattices of minimum~$5$ in which 
pairs of non-proportional minimal vectors have scalar product~$\pm 1$. 
In dimensions \hbox{$n\in[15$--$23]$} (and in some higher dimension) 
examples show that we have the strict inequality $t'_n>2n$,so that 
the highest values of $t'_n$ are attains on sets of angle $\arccos\frac 15$, 
whence the special interest of these lattices. 

Most of the lattices we shall consider (though not all) have been constructed 
using projections of a lattice of minimum~$4$ along a norm-$6$ vector, 
and using descending chains of cross-sections. 
We first consider such an exception.

\subsection{A 13-Dimensional Lattice} 

This is the lattice $L$ denoted by $C2\times \PSL(2,25):C2$ in \cite{WebN-S} 
(after its automorphism group). It has $s=26$, and defines the equiangular 
family referred to in Conjecture~\ref{conj8-13}. 
It has a unique (up to isometry) $14$-dimensional extension $\ov L$ 
with $s=28$, the value of~$t_{14}$, which extends uniquely wxactly 
up to dimension~$19$ under the condition that $s$ be as large as possible. 
We recover these lattices as the unique sequence of densest sections,  
for which Gram matrices denoted by $\Qa\!n$ can be read in 
\cite{WebM}, file Min5.GP. 
\newline{\small 
[The even sublattice $L_{e v}$ of $L$ deserves a remark: this is the best known 
example (found by Conway and Sloane, \cite{C-S}) of a lattice in its dimension 
having a large ``Berg\'e-Martinet invariant'' $\g'$, 
defined by $\g'(\Lb)=\big(\g(\Lb)\g(\Lb^*)\big)^{1/2}$.]}

\subsection{Lattices from the Leech Lattice} 

There is a unique orbit of norm-$6$ vectors in the Leech lattice $\Lb_{24}$, 
so that Theorem~\ref{thequiang} defines a unique lattice (up to isomorphism). 
Its set of minimal vectors consists in $276$ pairs $\pm x$, 
which are the projections along the corresponding norm-$6$ vector 
of vectors $x\in S(\Lb_{24}$ such that $v\cdot x=\pm 3$. 
This configuration is known as {\em the Witt design\/}.
Its automorphism group is $2\times\Co_3$. 

Consequently, for any pair $(L,v)$ of a relative lattice $L\subset\Lb_{24}$ 
of minimum~$4$ and of a norm-$6$ vector $v\in L$, provided that 
there exists an $x\in S(L)$ with $v\cdot x=\pm 3$, 
projection along $v$ defines a lattice whose set of minimal vectors 
supports an equiangular family contained in the Witt design. 

We have considered projections of various lattices contained in the Leech 
lattice in dimensions $14-24$ taken from \cite{WebM}, file Lambda.gp. 
The best results were obtained using Conway-Sloane's 
{\em laminated lattices\/} $\Lb_n$ or successive sections of such lattices. 
The file Min5.gp contains series of Gram matrices 
$\Qb\!n$, $16\ge n\ge 8$ and $\Qc\!n$, $23\ge n\ge 8$, starting with Gram 
matrices for projections of $\Lb_{17}$ and $\Lb_{24}$, respectively. 

The table below displays the known values of $t_n$ (taken from \cite{GSY} 
and the largest values found on lattices for equiangular families,  
using $\Qa\!n$ and $\Qb\!n$ for $n=14$, $\Qb\!n$ for $n=15,16$,  
and $\Qc\!n$ for 
\linebreak 
$17\le n\le 23$. 

\medskip 

\vbox{\small  
\ctl{Table for dimensions $14$ -- $23$} 

\medskip 
\begin{tabular}{|c|c|c|c|c|c|c|c|c|c|c|} 
\hline
$n$ & $14$ & $15$ & $16$ & $17$ & $18$ & $19$ & $20$ & $21$ & $22$ & $23$ \\ 
\hline 
\hline 
$t_n$ & $28$ & $36$ & $40$ & $48$ & $57-59$ & $72-74$ & $90-94$ & $126$ & 
$176$ & $276$ \\ 
\hline 
$lat\ge$ & $28$ & $36$ & $38$ & $48$ & $56$ & $72$ & $90$ & $126$ & 
$176$ & $276$ \\ 
\hline 
\end{tabular} 
} 

\medskip 

Inspection of the second line of this table clearly shows that we have 
$t'_n=t_n$ for all $n\in[14,23]$, and inspection of the third line 
shows that this values are attained on minimaL vectors of lattices 
for $n=14$, $15$, $17$ and \hbox{$21$--$23$}, 
As for the remaining values of $n$ I conjecture: 

\begin{conj} \label{conj14-24} 
For $n=16$, $18$, $19$ and~$20$, the largest number of lines produced 
by the minimal vectors of a lattice is $38$, $56$, $72$ and $90$, 
respectively. 
\end{conj} 

We now prove two complements, first for $n=14$, then for $n=18$, the latter 
answering a question raised in \cite{G}, Subsection~1.3.2. 

\begin{prop} \label{propdim14} 
The sets of equiangular lines afforded by $S(\Qa_{14})$ and $S(\Qb_{14})$ 
are not isometric. 
\end{prop} 

\begin{proof} 
The construction of $S(\Qa_{14})$ from $S(\Qa_{13})$ shows that removing two 
convenient lines from $S(\Qa_{14})$ produces a set of rank~$13$ 
(indeed, in a unique way), whereas one must remove at least four lines 
from $S(\Qb14))$ to obtain a set of rank~$\le 13$. 
\end{proof} 

In \cite{GSY} the authors construct four sets of $57$ vectors which, 
when rescaled to norm~$5$ have pairwise scalar products~$\pm 1$. 
I have checked that in all examples the sublattice 
they generate in $\Z^{18}$ has a minimum $m\le 4$ (and contains norm~$5$ 
vectors with pairwise scalar products not~$\pm 1$). 
Since any subset of norm~$5$ vectors in $\Qa_{23}$ generate a lattice 
of minimum~$5$ in its span, we have: 

\begin{prop} \label{propdim18} 
The four equiangular systems above are not contained in the Witt design. 
\qed 
\end{prop}

\subsection{Beyond dimension 23} 

The arguments used to prove Proposition~\ref{prop2-13}, (3) 
(relying on results of \cite{K-Y} together with the relative bound) prove: 

\begin{prop} \label{prop24-41} 
In the range $24\le n\le41$, $t'_n$ is strictly smaller than~$t_n$ 
(equal to $276$). 
\qed 
\end{prop} 

In analogy with what was observed in Section~\ref{secroot} 
for angle~$\arccos\frac13$, we expect that $\arccos\frac15$ should play 
a major r\^ole in dimensions, say, $24$ to~$50$. 
However, whereas lattices contained in the Leech lattice constitute a rich 
source of even lattices of minimum~$4$, our knowledge of such lattices 
in larger dimensions is poor. I have carried out some experimentation 
(far from being exhaustive) on the even sublattices of unimodular lattices 
of minimum~$3$, using the Bacher-Venkov classification in dimensions $27$
and~$28$ (see the file unimod23to28.gp.gp in\cite{WebM}). 
Using projections of the even part of the lattice denoted there by $o27b1$ 
we found lattices with $(n,s)=(26,82)$ and $(25,108)$, and then 
$(n,s)=(24,100)$ by cross-sections. 
Gram matrices can be downloaded from Part~4 of Min5.gp. 

\medskip 

Also no infinite series of lattices having 
a fixed minimum $m\ge 4$ are known (by contrast with minima $m=2$ or~$3$ 
for which we can use root lattices). 
This leaves wide open the following question: 

\begin{quest} \label{quesasympt} 
Do there exist for each odd $m\ge 5$ infinite series $L_n$ of lattices 
with minimum $m$ and pairwise scalar products $\pm 1$ on~$S(L_n)$ 
such that $s\sim\frac{m+1}{m-1}\,n$ for $n\to\infty\,$? 
{\small 
(If not the exact bound $\lfloor\frac{m+1}{m-1}\,(n-1)\rfloor$ 
of~\cite{JTYZZ}.)} 
\end{quest} 

I have also considered angles $\arccos\frac 17$, using projections 
of lattices of minimum~$6$. The only example which deserves to be mentioned 
is due to G.~Nebe (\cite{Ne}, answering a question of mine), the proof of which 
relies on the theory of modular forms: 
{\em applying Theorem~\ref{thequiang} to an even, unimodular lattice 
of minimum~$6$ yields an equiangular family of $100$~lines\/}.




\end{document} 

I have carried out some experimentation on even sublattices of unimodular 
lattices in dimensions~$\le 28$. If $v$ is a parity vector of minimal norm 
for such a lattice, we have $N(v)\equiv n\mod 8$ and (Elkies) $N(v)\le n-16$, 
so that for $25\le n\le 41$, this norm may be $n-16$ or $n-24$. 
Corresponding lattices were called by Venkov 
of {\em general type\/} and of {\em exceptional type\/}, respectively. 
Here are the numbers of each of the types, after a classification 
due to Bacher and Venkov: